\documentclass[12pt]{amsart}
\usepackage[margin=1.5in]{geometry}
\usepackage{amsmath, amssymb, amsfonts}
\usepackage[foot]{amsaddr}
\usepackage{color}

\usepackage{multirow, arydshln}
\usepackage{mathrsfs,extpfeil,enumerate}

\newcommand{\R}{\mathbb{R}}
\newcommand{\C}{\mathbb{C}}
\newcommand{\N}{\mathbb{N}}
\newcommand{\Z}{\mathbb{Z}}

\newcommand{\x}{\mbox{\boldmath $x$}}
\newcommand{\y}{\mbox{\boldmath $y$}}

\definecolor{LtBlue}{cmyk}{1,0,0,0}

\newtheorem{thm}{Theorem}[section]
\newtheorem{lemma}[thm]{Lemma}
\newtheorem{ex}[thm]{Example}
\newtheorem{conj}[thm]{Conjecture}
\newtheorem{prop}[thm]{Proposition}
\newtheorem{cor}[thm]{Corollary}
\newtheorem{question}[]{Question}

\numberwithin{equation}{section}

\begin{document}

\title[Complex equiangular lines from mutually unbiased bases]{Constructions of complex equiangular lines from mutually unbiased bases}

\author{Jonathan Jedwab and Amy Wiebe}
\address{Department of Mathematics, Simon Fraser University, 8888 University Drive, Burnaby, BC, Canada, V5A 1S6}
\date{21 August 2014 (revised 4 March 2015)}
\email{jed@sfu.ca, awiebe@sfu.ca}
\thanks{J. Jedwab is supported by an NSERC Discovery Grant. A. Wiebe was supported by an NSERC Canada Graduate Scholarship.}

\begin{abstract} A set of vectors of equal norm in $\C^d$ represents equiangular lines if the magnitudes of the Hermitian inner product of every pair of distinct vectors in the set are equal. The maximum size of such a set is $d^2$, and it is conjectured that sets of this maximum size exist in $\C^d$ for every $d \geq 2$. We take a combinatorial approach to this conjecture, using mutually unbiased bases (MUBs) in the following 3 constructions of equiangular lines:
\begin{enumerate}
\item{adapting a set of $d$ MUBs in $\C^d$ to obtain $d^2$ equiangular lines in $\C^d$,}
\item{using a set of $d$ MUBs in $\C^d$ to build $(2d)^2$ equiangular lines in $\C^{2d}$,}
\item{combining two copies of a set of $d$ MUBs in $\C^d$ to build $(2d)^2$ equiangular lines in $\C^{2d}$.}
\end{enumerate}
For each construction, we give the dimensions $d$ for which we currently know that the construction produces a maximum-sized set of equiangular lines. 
\end{abstract}

\maketitle

\section{Introduction} \label{S:Intro}

\allowdisplaybreaks

Equiangular lines have been studied for over 65 years \cite{Haantjes}, and their construction remains ``[o]ne of the most challenging problems in algebraic combinatorics'' \cite{Khatirinejad}. In particular, the study of equiangular lines in complex space has intensified recently, as its importance in quantum information theory has become apparent \cite{Appleby-Clifford, Grassl-tomography, Renes, Scott-Grassl}. It is well-known that the maximum number of equiangular lines in $\C^d$ is $d^2$ \cite{DGS-bounds, Godsil-Roy}. Zauner \cite{Zauner-thesis} conjectured 15 years ago that this upper bound can be attained for all $d\geq2$. This conjecture is supported by exact examples in dimensions 2, 3 \cite{DGS-bounds, Renes}, 4, 5 \cite{Zauner-thesis}, 6 \cite{Grassl-dim6}, 7 \cite{Appleby-Clifford, Khatirinejad}, 8 \cite{Appleby-imprimitivity, Grassl-tomography, Hoggar-two, Scott-Grassl}, 9--15 \cite{Grassl-tomography, Grassl-slides, Grassl-computing}, 16 \cite{Appleby-monomial}, 19 \cite{Appleby-Clifford, Khatirinejad}, 24 \cite{Scott-Grassl}, 28 \cite{Appleby-imprimitivity}, 35 and 48 \cite{Scott-Grassl}, and by examples with high numerical precision in all dimensions $d\leq 67$ \cite{Renes, Scott-Grassl}. However, Scott and Grassl \cite{Scott-Grassl} note that ``[a]lthough our confidence in its truth has grown considerably, we seem no closer to a proof of Zauner's conjecture than Gerhard Zauner was at the time of his doctoral dissertation.''

M.~Appleby \cite{Appleby-abstract} observed in 2011: ``In spite of strenuous attempts by numerous investigators over a period of more than 10 years we still have essentially zero insight into the structural features of the equations [governing the existence of a set of $d^2$ equiangular lines in $\C^d$] which causes them to be soluble. Yet one feels that there must surely be such a structural feature \dots (one of the frustrating features of the problem as it is currently formulated is that the properties of an individual [set of $d^2$ equiangular lines in $\C^d$] seem to be highly sensitive to the dimension).'' In light of this difficulty, one of the aims of this paper is to illuminate structural features of sets of equiangular lines that are common across several dimensions.  

There are many papers addressing both the topic of maximum-sized sets of equiangular lines and that of mutually unbiased bases \cite{Appleby-prime, Appleby-monomial, Beneduci-operational, Bengtsson-Eddington, Bengtsson-KS, Greaves, Kibler, Wootters}. In 2005, Appleby \cite{Appleby-prime} even stated: ``There appear to be some intimate connections [between the study of complex equiangular lines and] the theory of mutually unbiased bases \dots''. Nonetheless, in this paper we show that there appear to be still deeper connections between these two objects than previously recognized. 

The remainder of this paper is organized as follows. In Section~\ref{S:defns} we define the major objects that we use in the rest of the paper; in Section~\ref{S:Zauner} we give an overview of the standard method of construction of equiangular lines; in Sections~\ref{S:LinesfromMUBs}--\ref{S:Lblocks} we describe three new construction methods of equiangular lines from MUBs, including examples from the dimensions $d$ for which we currently know they succeed; and in Section~\ref{S:conclusion} we give some concluding remarks.

\section{Definitions} \label{S:defns}

We now introduce the main objects of study.

A line through the origin in $\C^d$ can be represented by a nonzero vector $\x\in\C^d$ which spans it. 
A set of $m\geq 2$ distinct lines in $\C^d$, represented by vectors $\x_1,\ldots,\x_m$, is {\em equiangular} if there is some real constant $c$ such that 
\begin{equation*}  \frac{{|\langle \x_j,\x_k\rangle|}}{||\x_j||\cdot||\x_k||} = c \hspace{10pt}\mbox{ for all } j\neq k, \end{equation*}
where $\langle \x,\y\rangle$ is the standard Hermitian inner product in $\C^d$ and $||\x|| = \sqrt{|\langle \x, \x\rangle|}$ is the norm of $\x$. 
We simplify notation by always taking $\x_1,\ldots,\x_m$ to have equal norm, and then it suffices that there is a constant $a$ such that 
\begin{equation} |\langle \x_j,\x_k\rangle| = a \hspace{10pt}\mbox{ for all } j\neq k. \label{EQ:equiang2} \end{equation}
Furthermore, if each vector has unit norm, then we will refer to $|\langle \x_j,\x_k\rangle|$ as the {\em angle} between $\x_j$ and $\x_k$ (although this value is strictly the cosine of the angle). 

It is known that there can be at most $d^2$ equiangular lines in $\C^d$~\cite{DGS-bounds}. This is a specific instance of more general results obtained by Delsarte, Goethals and Seidel~\cite{DGS-bounds} using Jacobi polynomials. They found special bounds on the number of lines with a small set of angles that can exist in $\C^d$ when the angle values are specified (see~\cite[Table~I]{DGS-bounds}), as well as absolute bounds on the number of lines with a small set of angles that can exist in $\C^d$ without specifying angle values (see~\cite[Table~II]{DGS-bounds}). They also noted that if $\{\x_1,\ldots, \x_{d^2}\}$ is a set of unit vectors representing a maximum-sized set of complex equiangular lines, then the value of $a$ in (\ref{EQ:equiang2}) is determined.

\begin{prop} Let $\{\x_1,\ldots, \x_{d^2}\}$ be a set of unit vectors representing equiangular lines in $\C^d$. Then 
$$|\langle \x_j,\x_k\rangle| = \frac{1}{\sqrt{d+1}}$$
for all $j\neq k$. \label{PROP:angle}
\end{prop}

The value $1/\sqrt{d+1}$ given in Proposition~\ref{PROP:angle} can be determined by taking $s=1, \varepsilon=0$ over~$\C$ in \cite[Table II]{DGS-bounds}. An alternative self-contained proof using linear algebra is given in \cite[Proposition 9]{Wiebe-thesis}, following the method described by Godsil \cite{Godsil-slides}.

A basis for $\C^d$ is called {\em orthogonal} if the inner product of any two distinct basis elements is 0. Let $\{\x_1,\ldots, \x_d\},\{\y_1,\ldots,\y_d\}$ be two distinct orthogonal bases for $\C^d$. They are called {\em unbiased bases} if 
\begin{equation} \frac{|\langle \x_j,\y_k\rangle|}
{||\x_j||\cdot||\y_k||} = \frac{1}{\sqrt{d}} \label{EQ:MUBdef} 
\quad \mbox{for all $j,k$}.
\end{equation}
A set of orthogonal bases is a set of {\em mutually unbiased bases (MUBs)} if all pairs of distinct bases are unbiased.

\begin{ex} Consider the following orthogonal bases for $\C^2$: 
\begin{eqnarray*}
B_1 = \left\{ \begin{array}{@{\; (}cc@{)\;}}
1 & 0 \\
0 & 1
\end{array}\right\}
&
B_2 = \left\{ \begin{array}{@{\; (}cc@{)\;}}
1 & 1 \\
1 & -1
\end{array}\right\}
&
B_3 = \left\{ \begin{array}{@{\; (}cc@{)\;}}
1 & i \\
1 & -i
\end{array}\right\}.
\end{eqnarray*}
Then for $\x,\y$ in distinct bases we have 
\begin{eqnarray*}
\frac{|\langle \x,\y\rangle|}{||\x||\cdot||\y||} 
	& = & \begin{cases} \frac{1}{1\cdot\sqrt{2}} & \text{ for one of } \x,\y\in B_1\\[3pt]
		\frac{\sqrt{2}}{\sqrt{2}\cdot\sqrt{2}} & \text{ for }\x,\y\notin B_1 \end{cases} \\
	& = & \frac{1}{\sqrt{2}},
\end{eqnarray*}
satisfying $(\ref{EQ:MUBdef})$, so $\{B_1,B_2,B_3\}$ is a set of $3$ MUBs in $\C^2$. 
\label{EX:MUB}
\end{ex}

An upper bound on the number of MUBs in $\C^d$ is $d+1$ \cite[Table I]{DGS-bounds} (using $\alpha=1/d,\beta=0$ over $\C$). An alternative proof of this bound is given in \cite[Proposition 16]{Wiebe-thesis} using linear algebra, following the method described by Bandyopadhyay {\em et al.} \cite{Bandyopadhyay}.
As with equiangular lines, the central question concerning MUBs is whether this bound can be attained in all dimensions. In contrast to the situation for equiangular lines, there seems to be more doubt that this is possible. Bengtsson \cite{Bengtsson-Eddington} in 2011 observed: ``The belief in the community is that a complete set of $[d+1]$ MUB[s] does not exist for general $[d]$, while the [maximum-sized sets of equiangular lines] do.'' 
However, it is known that this upper bound for MUBs is attainable in prime power dimensions $d$ using the method of Godsil and Roy \cite{Godsil-Roy} which we follow here. 

Let $G$ be a group of order $mn$, containing a normal subgroup $N$ of order~$n$. A {\em $(m,n,k,\lambda)$-relative difference set (RDS)}\index{relative difference set}\index{RDS} in $G$ relative to $N$ is a subset $R\subset G$ of size $k$, such that the multiset $$\{r_1r_2^{-1} : r_1,r_2\in R, r_1\neq r_2\}$$ contains each element of $G\backslash N$ exactly $\lambda$ times and does not contain any elements of $N$.

\begin{ex} Let $G$ be the abelian group of order $16$ given by $\langle x\rangle\times\langle y\rangle$, with $x^4=y^4=1$. Let $N$ be the subgroup $\langle x^2\rangle\times\langle y^2\rangle$ of order $4$. Then $R =\{1,x,y,x^3y^3\}$ is a $(4,4,4,1)$-RDS in $G$ relative to $N$.
\label{EX:RDS4} \end{ex}

A {\em character} of a finite abelian group $G$ is a map $\chi:G\to\C$ which is a group homomorphism. If $G$ has order $v$, then there are $v$ characters, each of which maps the elements of $G$ to roots of unity. These characters form a group $G^*$ which is isomorphic to $G$. (See Pott \cite{Pott}, for example, for background on the use of characters to study relative difference sets.)

\begin{thm}[Theorem 4.1, \cite{Godsil-Roy}] Let $R=\{r_1,\ldots, r_d\}$ be a $(d,d,d,1)$-RDS in an abelian group $G$ relative to some subgroup of order $d$. Then the set of vectors $$\{(\chi(r_1),\ldots, \chi(r_d)) : \chi\in G^*\}$$
comprises a set of $d$ MUBs in $\C^d$.
\label{THM:RDStoMUBs} \end{thm}

\begin{ex} Take $R$ to be the RDS of Example~$\ref{EX:RDS4}$. The group $G$ has $16$ characters given by $\{\chi_{j,k}: j,k\in\Z_4\}$, where $\chi_{j,k}(x) = i^j, \chi_{j,k}(y)=i^k$. Using the construction of Theorem~$\ref{THM:RDStoMUBs}$ we get the following $4$ MUBs $B_1, B_2, B_3, B_4$ in $\C^4$:
\begin{displaymath}
{\renewcommand{\arraystretch}{1.1}
\begin{array}{c|c@{(}cccc@{)\hspace{5pt}}c}
	 	&& 1 & x & y & x^3y^3 \\ \hline
	 \chi_{0,0} && 1 & 1 & 1 & 1 & \multirow{4}{*}{$\left. \rule{0pt}{32pt}\right\} B_1 $} \\
	 \chi_{0,2} && 1 & 1 & -1 & -1 \\
	 \chi_{2,0} && 1 & -1 & 1 & -1 \\
	 \chi_{2,2} && 1 & -1 & -1 & 1 \\[2pt] \hdashline[2pt/6pt]
	 \chi_{0,1} && 1 & 1 & i & -i & \multirow{4}{*}{$\left. \rule{0pt}{32pt}\right\} B_2 $} \\
	 \chi_{0,3} && 1 & 1 & -i & i \\
	 \chi_{2,1} && 1 & -1 & i & i \\
	 \chi_{2,3} && 1 & -1 & -i & -i \\[2pt] \hdashline[2pt/6pt]
	 \chi_{1,0} && 1 & i & 1 & -i & \multirow{4}{*}{$\left. \rule{0pt}{32pt}\right\} B_3 $} \\
	 \chi_{1,2} && 1 & i & -1 & i \\
	 \chi_{3,0} && 1 & -i & 1 & i \\
	 \chi_{3,2} && 1 & -i & -1 & -i \\[2pt] \hdashline[2pt/6pt]
	 \chi_{1,1} && 1 & i & i & -1 & \multirow{4}{*}{$\left. \rule{0pt}{32pt}\right\} B_4. $} \\
	 \chi_{1,3} && 1 & i & -i & 1 \\
	 \chi_{3,1} && 1 & -i & i & 1 \\
	 \chi_{3,3} && 1 & -i & -i & -1 
\end{array}
}
\end{displaymath} \label{EX:MUBs4}
\end{ex}

We note that to attain the upper bound of $d+1$ MUBs in $\C^d$ when $d$ is a prime power, we can include the standard basis with the $d$ bases obtained from Theorem~\ref{THM:RDStoMUBs}.

When $d$ is not a prime power, the smallest dimension for which a set of $d+1$ MUBs in $\C^d$ is not known is $d=6$. No one has even found 4 MUBs in $\C^6$; furthermore, the existence of sets of 3 MUBs which are provably not extendable to 4 leads some researchers to suspect that it may not be possible to find more than 3 MUBs in $\C^6$ \cite{Bengtsson-three-ways}.

\section{Zauner's Construction} \label{S:Zauner}

Zauner \cite{Zauner-thesis} was the first to conjecture that maximum-sized sets of equiangular lines exist in all dimensions. Along with this conjecture, he presented a construction for such sets of lines which has become the standard construction in the area. We give a brief overview of this construction here. 

In most of the literature regarding complex equiangular lines, maximum-sized sets of equiangular lines are constructed as the orbit of a {\em fiducial vector} under the action of a group of matrices. Zauner's thesis \cite{Zauner-thesis} describes both the group to use and where to find an appropriate fiducial vector. 

Let $\omega = e^{2\pi i/d}$ and 
$$U = \left[\begin{array}{cccccc} 
	1 & 0 & 0 & \cdots & 0  \\
	0 & \omega & 0 & \cdots & 0  \\
	0 & 0 & \omega^2 & \cdots & 0  \\
	\vdots &  & & \ddots & \vdots \vspace{2pt}\\
	0 & 0 & 0 & \cdots & \omega^{d-1}
	\end{array}\right]
	\mbox{\hspace{10pt} and \hspace{10pt}}
V = \left[\begin{array}{cccccc} 
	0 & 1 & 0 & 0 & \cdots & 0  \\
	0 & 0 & 1 & 0 & \cdots & 0  \\
	\vdots & &  & \ddots & & \vdots \\
	\\
	0 & 0 & 0 & 0 & \cdots & 1 \\
	1 & 0 & 0 & 0 & \cdots & 0
	\end{array}\right].$$
Matrices $U,V$ generate a group of matrices known as the Weyl-Heisenberg group \cite{Weyl} (see \cite{Bengtsson-Eddington, Scott-Grassl}, for example, for an overview). Modulo its center, the group is isomorphic to $\Z_d\times\Z_d$, and the elements $U^jV^k$ for $j,k\in\Z_d$ form a set of coset representatives for the center in this group.

Next define a $d\times d$ matrix $Z_u=(z_{jk})$, for $j,k\in\{0,\ldots,d-1\},$\index{$Z_u$} by 
\begin{equation*} z_{jk} = \frac{e^{\pi i(d-1)/12}}{\sqrt{d}}e^{\pi i(2jk+(d+1)k^2)/d}. \end{equation*}
Then $Z_u$ is a unitary matrix (often referred to as {\em Zauner's unitary}\index{Zauner's unitary}) which satisfies
$$Z_u^3  =  I_d $$
and normalizes the Weyl-Heisenberg group.

Zauner's full conjecture is the following:
\begin{conj} For each $d\geq 2$, there exists a set of $d^2$ equiangular lines in $\C^d$ that is constructed as 
$$\{A\x^T:A\in G\},$$
where $G$ is the Weyl-Heisenberg group and $\x^T$ is some eigenvector of $Z_u$ having eigenvalue $1$. \label{CONJ:ZaunerFull}
\end{conj} 

In all dimensions where a maximum-sized set of equiangular lines is known, there is a set constructed as in Conjecture~\ref{CONJ:ZaunerFull}. Furthermore, almost all known maximum-sized sets of equiangular lines can be constructed in this way. 

Notice that Conjecture~\ref{CONJ:ZaunerFull} does not state that every eigenvector of eigenvalue 1 will produce a maximum-sized set of equiangular lines. For further details on the computational methods by which appropriate eigenvectors are found and the difficulty of doing so, we refer the reader to \cite{Grassl-tomography, Grassl-computing, Scott-Grassl}.

\begin{ex} $d=4$ \cite[p.~62]{Zauner-thesis}: 
In dimension $4$, the generators of the Weyl-Heisenberg group and Zauner's unitary are as follows:
$$U = \left[\begin{array}{@{\,}cccc@{\,}} 
	1 & 0 & 0 & 0  \\
	0 & i & 0 &  0  \\
	0 & 0 & -1 & 0  \\
	0 & 0 & 0 &  -i
	\end{array}\right],\,
V = \left[\begin{array}{@{\,}cccc@{\,}} 
	0 & 1 & 0 & 0    \\
	0 & 0 & 1 & 0    \\
	0 & 0 & 0 & 1 	\\
	1 & 0 & 0 & 0 
	\end{array}\right],\, {\renewcommand{\arraystretch}{1.5} 
Z_u = \frac{1}{2}\left[\begin{array}{@{\,}cccc@{\,}} 
	\frac{1+i}{\sqrt{2}} & -i & -\frac{1+i}{\sqrt{2}} & -i    \\
	\frac{1+i}{\sqrt{2}} & 1 & \frac{1+i}{\sqrt{2}} & -1    \\
	\frac{1+i}{\sqrt{2}} & i & -\frac{1+i}{\sqrt{2}} & i    \\
	\frac{1+i}{\sqrt{2}} & -1 & \frac{1+i}{\sqrt{2}} & 1 
	\end{array}\right].
	}$$
A fiducial vector is given by the eigenvector 
\[
\left( \begin{array}{c} x_0 \\ x_1 \\ x_2 \\ x_3 \end{array} \right) = 
\frac{1}{2\sqrt{6}}\sqrt{3-\frac{3}{\sqrt{5}}} 
\left ( \begin{array}{c} \omega+1 \\ i \\ \omega-1 \\ i \end{array} \right) + 
\frac{1}{2\sqrt{2}}\sqrt{1+\frac{3}{\sqrt{5}}} 
\left ( \begin{array}{c} 0 \\ \omega \\ 0 \\ -\omega \end{array} \right)
\]
of $Z_u$, where $\omega=e^{i\pi/4}$.
Then $16$ equiangular lines in $\C^4$ are given by
\begin{displaymath}
{\renewcommand{\arraystretch}{1.1} 
\begin{array}{@{}rrrr@{}}
(x_0 & x_1 & x_2 & x_3)\; \\
(x_0 & ix_1 & -x_2 & -ix_3)\; \\
(x_0 & -x_1 & x_2 & -x_3)\; \\
(x_0 & -ix_1 & -x_2 & ix_3)\; \\
(x_1 & x_2 & x_3 & x_0)\; \\
(x_1 & ix_2 & -x_3 & -ix_0)\; \\
(x_1 & -x_2 & x_3 & -x_0)\; \\
(x_1 & -ix_2 & -x_3 & ix_0)\; \\
(x_2 & x_3 & x_0 & x_1)\; \\
(x_2 & ix_3 & -x_0 & -ix_1)\; \\
(x_2 & -x_3 & x_0 & -x_1)\; \\
(x_2 & -ix_3 & -x_0 & ix_1)\; \\
(x_3 & x_0 & x_1 & x_2)\; \\
(x_3 & ix_0 & -x_1 & -ix_2)\; \\
(x_3 & -x_0 & x_1 & -x_2)\; \\
(x_3 & -ix_0 & -x_1 & ix_2).
\end{array}
}
\end{displaymath}
\label{EX:WH4} \end{ex}

Now we describe three new constructions of equiangular lines in $\C^d$, each of which involves a set of MUBs. 

\section{Construction 1} \label{S:LinesfromMUBs}

The underlying structure of the 16 equiangular lines in $\C^4$ of Example~\ref{EX:WH4} seems strictly tied to the Weyl-Heisenberg group and requires complicated constants.
However, Appleby {\em et al}. \cite{Appleby-monomial} recently reinterpreted Zauner's construction (as described in \S\ref{S:Zauner}), leading to a new example in dimension 4 with simpler constants. We will show how the resulting simplified set of lines, given in the following example, has additional underlying combinatorial structure.

\begin{ex} A set of $16$ vectors representing equiangular lines in $\C^4$, as constructed in \cite{Appleby-monomial}, is
\begin{displaymath}
\begin{array}{r@{\;}cccc@{\;}l}
(&\sqrt{2+\sqrt{5}} & 1 & 1 & 1&) \\
(&\sqrt{2+\sqrt{5}} & 1 & -1 & -1&) \\
(&\sqrt{2+\sqrt{5}} & -1 & 1 & -1&) \\
(&\sqrt{2+\sqrt{5}} & -1 & -1 & 1&) \\[8pt]
(&1 & 1 & i\sqrt{2+\sqrt{5}} & -i&) \\
(&1 & 1 & -i\sqrt{2+\sqrt{5}} & i&) \\
(&1 & -1 & i\sqrt{2+\sqrt{5}} & i&) \\
(&1 & -1 & -i\sqrt{2+\sqrt{5}} & -i&)\\[8pt]
(&1 & i & 1 & -i\sqrt{2+\sqrt{5}}&) \\
(&1 & i & -1 & i\sqrt{2+\sqrt{5}}&) \\
(&1 & -i & 1 & i\sqrt{2+\sqrt{5}}&) \\
(&1 & -i & -1 & -i\sqrt{2+\sqrt{5}} &)\\[8pt]
(&1 & i\sqrt{2+\sqrt{5}} & i & -1&) \\
(&1 & i\sqrt{2+\sqrt{5}} & -i & 1&) \\
(&1 & -i\sqrt{2+\sqrt{5}} & i & 1&) \\
(&1 & -i\sqrt{2+\sqrt{5}} & -i & -1 &).
\end{array} \end{displaymath}
\label{EX:d=4}
\end{ex}

This set of vectors was also found by Belovs \cite{Belovs-thesis} in 2008 via another method and was known even earlier by Zauner (unpublished notes, 2005, referenced in \cite{Appleby-monomial}).  However, here we describe another construction of this set of equiangular lines which demonstrates that the underlying structure of this set can be interpreted as a set of 4 MUBs in $\C^4$ coming from a $(4,4,4,1)$-RDS as in Theorem~\ref{THM:RDStoMUBs}. The general construction is as follows.

We exploit the already constrained angles between vectors in a set of MUBs to produce sets of equiangular lines by allowing the multiplication of a single entry in each of the $d^2$ vectors by a constant. Specifically, let $B^R_1,\ldots, B^R_d$ be $d$ MUBs in $\C^d$ formed from a $(d,d,d,1)$-RDS $R$ according to Theorem~\ref{THM:RDStoMUBs}. Let $\pi\in S_d$ be a permutation of $\{1,\ldots,d\}$ (which we represent as an ordered list of images). Let $B^R_j(\pi,v)$ denote the set of vectors formed by multiplying entry $\pi(j)$ of each vector in $B_j^R$ by $v\in\C$. Let 
\begin{equation} L^R_d(\pi,v) = \bigcup_{j=1}^d B^R_j(\pi,v).\index{$L^R_d(\pi,v)$} \label{EQ:Lblock} \end{equation} 

\begin{ex} Take $B_1^R, \ldots, B_4^R$ to be the $4$ MUBs of Example~$\ref{EX:MUBs4}$.
Let $\pi = [1,3,4,2]$ and let $v\in\C$. Then $L^R_4(\pi,v)$ consists of vectors
\begin{displaymath}
{\renewcommand{\arraystretch}{1.1}
\begin{array}{@{(}cccc@{)\hspace{5pt}}c}
	v & 1 & 1 & 1  & \multirow{4}{*}{$\left. \rule{0pt}{30pt}\right\} B^R_1(\pi,v) $} \\
	v & 1 & -1 & -1 \\
	v & -1 & 1 & -1 \\
	v & -1 & -1 & 1 \\[5pt]
	1 & 1 & iv & -i  & \multirow{4}{*}{$\left. \rule{0pt}{30pt}\right\} B^R_2(\pi,v) $} \\
	1 & 1 & -iv & i \\
	1 & -1 & iv & i \\
	1 & -1 & -iv & -i \\[5pt] 
	1 & i & 1 & -iv  & \multirow{4}{*}{$\left. \rule{0pt}{30pt}\right\} B^R_3(\pi,v) $} \\
	1 & i & -1 & iv \\
	1 & -i & 1 & iv \\
	1 & -i & -1 & -iv \\[5pt] 
	1 & iv & i & -1& \multirow{4}{*}{$\left. \rule{0pt}{30pt}\right\} B^R_4(\pi,v) $} \\
	1 & iv & -i & 1 \\
	1 & -iv & i & 1 \\
	1 & -iv & -i & -1 
\end{array}
}
\end{displaymath}
and is a set of equiangular lines when $v\in\{ \pm \sqrt{2+\sqrt{5}}$, $\pm i\sqrt{2+\sqrt{5}}\}$. In fact, for $v = \sqrt{2+\sqrt{5}}$ it is the same set of lines as given in Example~$\ref{EX:d=4}$.
\label{EX:MUBLines4} \end{ex}

\begin{thm} For $d=2,3,4$ there exists a $(d,d,d,1)$-RDS $R$, a permutation $\pi\in S_d$ and a constant $v(d)\in\C$ such that $L^R_d(\pi,v(d))$
is a set of $d^2$ equiangular lines in $\C^d$. \label{THM:1constantconstruct}
\end{thm}

Example~\ref{EX:MUBLines4} together with the following two examples prove the theorem:

\begin{ex} $d=2:$ Take the $(2,2,2,1)$-RDS $R=\{1,x\}$ in $\langle x\rangle\cong \Z_4$ relative to $\langle x^2\rangle\cong\Z_2$, the permutation $\pi=[1,2]$ and the constant $v$. This gives a set of $2$ MUBs in $\C^2$ ($B_2,B_3$ from Example~$\ref{EX:MUB}$). Then $L_2^R(\pi,v)$ consists of vectors 
\begin{displaymath}\begin{array}{cc}
\begin{array}{@{\;(}cc@{)\hspace{3pt}}c}
v & 1 & \multirow{2}{*}{$\left. \rule{0pt}{12pt}\right\} B^R_1(\pi,v) $} \\
v & -1 \\[8pt]
1 & iv & \multirow{2}{*}{$\left. \rule{0pt}{12pt}\right\} B^R_2(\pi,v) $} \\
1 & -iv
\end{array}
\end{array}\end{displaymath}
which are equiangular for $v\in\left\{\frac{1}{2}(\sqrt{2}\pm\sqrt{6})\right.$, $-\frac{1}{2}(\sqrt{2}\pm\sqrt{6})$, $\frac{1}{2}i(\sqrt{2}\pm\sqrt{6})$, $\left.-\frac{1}{2}i(\sqrt{2}\pm\sqrt{6})\right\}$. \label{EX:MUBLines2}
\end{ex} 

\begin{ex} $d=3:$ Take the $(3,3,3,1)$-RDS $R = \{1,y,xy^2\}$ in $\langle x\rangle\times\langle y\rangle \cong\Z_3\times\Z_3$ relative to $\langle x\rangle \times \langle 1\rangle \cong \Z_3$. The resulting $3$ MUBs are as follows: 
{\renewcommand{\arraystretch}{1.1}
\begin{displaymath}
\begin{array}{c|c@{(}ccc@{)\;}l}
	 	&& 1 & y & xy^2 \\ \hline
	 \chi_{0,0} && 1 & 1 & 1 & \multirow{3}{*}{$\left.\rule{0pt}{22pt}\right\} B^R_1$} \\
	 \chi_{0,1} && 1 & \omega & \omega^2  \\
	 \chi_{0,2} && 1 & \omega^2 & \omega  \\[2pt] \hdashline[2pt/6pt]
	 \chi_{1,0} && 1 & 1 & \omega & \multirow{3}{*}{$\left.\rule{0pt}{22pt}\right\} B^R_2$} \\
	 \chi_{1,1} && 1 & \omega & 1  \\
	 \chi_{1,2} && 1 & \omega^2 & \omega^2  \\[2pt] \hdashline[2pt/6pt]
	 \chi_{2,0} && 1 & 1 & \omega^2  &  \multirow{3}{*}{$\left.\rule{0pt}{22pt}\right\} B^R_3$} \\
	 \chi_{2,1} && 1 & \omega & \omega  \\
	 \chi_{2,2} && 1 & \omega^2 & 1
\end{array}
\end{displaymath}}
where $\omega=e^{2\pi i/3}$.
Take the permutation $\pi = [1,2,3]$ and the constant $v$. Then $L_3^R(\pi,v)$ consists of vectors 
\begin{displaymath}
\begin{array}{@{(}ccc@{)\;}l}
	 v & 1 & 1  & \multirow{3}{*}{$\left.\rule{0pt}{20pt}\right\} B^R_1(\pi,v)$} \\
	 v & \omega & \omega^2   \\
	 v & \omega^2 & \omega \\[8pt]
	 1 & v & \omega  &  \multirow{3}{*}{$\left.\rule{0pt}{20pt}\right\} B^R_2(\pi,v)$} \\
	 1 & v\omega & 1   \\
	 1 & v\omega^2 & \omega^2 \\[8pt]
	 1 & 1 & v\omega^2 &   \multirow{3}{*}{$\left.\rule{0pt}{20pt}\right\} B^R_3(\pi,v)$} \\
	 1 & \omega & v\omega   \\
	 1 & \omega^2 & v 
\end{array}
\end{displaymath}
which are equiangular for $v=0$.
\label{EX:MUBLines3}
\end{ex}

It turns out that in Examples~\ref{EX:MUBLines2} ($d=2$) and \ref{EX:MUBLines3} ($d=3$), every choice of permutation $\pi\in S_d$ will produce a maximum-sized set of equiangular lines with the given constant(s) $v$. However, in Example~\ref{EX:MUBLines4} ($d=4$), the choice of permutation becomes important, as only 8 of the 24 possible permutations, namely, $[1,3,4,2]$,  $[1,4,2,3]$,  $[2,3,1,4]$,  $[2,4,3,1]$,  $[3,1,2,4]$,  $[3,2,4,1]$,  $[4,1,3,2]$, and  $[4,2,1,3]$, admit a constant $v$ which results in a set of equiangular lines. The occurrence of these permutations can be explained by the following theorem, which we state here without proof. 

\begin{thm}[Theorem 46, \cite{Wiebe-thesis}] Fix a $(4,4,4,1)$-RDS $R$ in $\Z_4\times\Z_4$. The vectors of $L^R_4(\pi,\sqrt{2+\sqrt{5}})$ comprise $16$ equiangular lines in $\C^4$ if $\pi$ is such that the inner product between each pair of distinct vectors in $L^R_4(\pi,0)$ has magnitude~$\sqrt{2}$. \label{THM:d=4_1constant} \end{thm}

Theorem~\ref{THM:d=4_1constant} suggests that as we increase the dimension $d$ there may be additional restrictions on when this construction produces maximum-sized sets of equiangular lines, which leads us to ask the following question:

\begin{question} Does the construction used in Theorem~$\ref{THM:1constantconstruct}$ produce maximum-sized sets of equiangular lines for some $d>4$? \label{Q:1constantconstruct} \end{question}

And more generally, 

\begin{question} Can we transform a set of $d$ MUBs in $\C^d$ into $d^2$ equiangular lines in $\C^d$ through multiplication by constants? \label{Q:MUBstoLines} \end{question}

Conversely, we wonder if there is a complementary construction to the one we have described, which could be used to extract MUBs from maximum-sized sets of equiangular lines other than those given in Examples~\ref{EX:MUBLines4}, \ref{EX:MUBLines2} and \ref{EX:MUBLines3}.


\begin{question} Can we transform a set of $d^2$ equiangular lines in $\C^d$ into $d$ MUBs in $\C^d$ through multiplication by constants? \label{Q:LinestoMUBs} \end{question}

We now derive a necessary condition that could assist in answering Question~\ref{Q:1constantconstruct}.
We observe that the construction used in Theorem~$\ref{THM:1constantconstruct}$ will always produce {\em almost flat} vectors; that is, vectors having all but one entry of equal magnitude. Furthermore, if this set of vectors represents a set of equiangular lines, then the following lemma determines the single exceptional magnitude:

\begin{lemma} Let $R$ be a $(d,d,d,1)$-RDS in an abelian group. Let $\pi\in S_d$ and $v\in\C$. Suppose that $L^R_d(\pi,v)$ is a set of $d^2$ equiangular lines. Then the magnitude of $v$ is $\sqrt{2\pm\sqrt{d+1}}$. \end{lemma}

\begin{proof} 
Let $\x,\y$ be distinct lines of $L^R_d(\pi,v)$ originating from the same basis. Since the original vectors are orthogonal, the inner product of $\x$ and $\y$ is $\xi(|v|^2-1)$ for some root of unity $\xi$. Each of $\x$ and $\y$ has norm $\sqrt{d-1+|v|^2}$. By Proposition~\ref{PROP:angle}, we must have 
$$
\frac{1}{\sqrt{d+1}} = \frac{|\langle \x,\y\rangle|}{||\x||\cdot||\y||} = \frac{\big||v^2|-1\big|}{{d-1+|v|^2}},
$$
which is easily solved to find $|v|^2 = 2\pm\sqrt{d+1}$. 
\end{proof}

Though this construction is notably different than Zauner's, we observe that having an exceptional magnitude of $\sqrt{2\pm\sqrt{d+1}}$ is equivalent to the necessary condition given in \cite[6.4.1.\ Lemma]{Roy-thesis} for constructing maximum-sized sets of almost flat equiangular lines using Zauner's construction. 

As a final observation about this construction, notice that MUBs formed as in Theorem~\ref{THM:RDStoMUBs} have elements all of whose entries lie on the complex unit circle. Thus if we take a single basis from the set and write its elements as the rows of a matrix $H$, then this matrix will satisfy the equation
\begin{equation} HH^\dagger = dI_d \end{equation}
(where $H^\dagger$ is the conjugate transpose of $H$), which means $H$ is a complex Hadamard matrix of order $d$. 
From this observation, we can consider the single Hadamard matrix construction of \cite{Jedwab-Wiebe-simple} for 64 equiangular lines in $\C^8$ as an example of a construction relying on MUBs. This now links the construction of maximum-sized sets of equiangular lines in dimensions $2,3,4$ and $8$ via MUBs.

\section{Construction 2}\label{S:d8}

We now examine another example of a maximum-sized set of equiangular lines whose connection to MUBs has not previously been recognized. 

One of the first dimensions for which a set of $d^2$ equiangular lines in $\C^d$ was discovered was $d=8$. In  1981, Hoggar gave a construction for 64 lines as the (complexified vectors associated to the) diameters of a quaternionic polytope~\cite{Hoggar-two}. (See \cite{Hoggar-t-designs, Hoggar-64} for more details on this construction.)
 
This set of lines was reinterpreted by Zauner~\cite{Zauner-thesis}. It was also reinterpreted by Godsil and Roy~\cite{Godsil-Roy} using the following variation of Zauner's construction  method of \S\ref{S:Zauner}.

Recall that in dimension 2, the Weyl-Heisenberg group has generators 
$$U = \left[\begin{array}{cc} 1& 0 \\ 0 & -1\end{array}\right], \hspace{20pt}
V = \left[\begin{array}{cc} 0 & 1 \\ 1 & 0\end{array}\right] $$
and a set of coset representatives for its center is $\{I_4,U,V,UV\}$. Now consider the 3-fold tensor product $G$ of this set: it has 64 elements given by $G = \{A\otimes B\otimes C: A,B,C\in \{I_4,U,V,UV\}\}$ with each element being an $8\times 8$ matrix. Let 
$$\x = \left(0,0, \frac{1+i}{\sqrt{2}}, \frac{1-i}{\sqrt{2}}, \frac{1+i}{\sqrt{2}},- \frac{1+i}{\sqrt{2}},0,\sqrt{2}\right).$$
Then 64 equiangular lines in $\C^8$ are given by $\{A\x^T : A\in G\}$.

There are several operations that map a set of equiangular lines to an equivalent set of equiangular lines, including:
\begin{enumerate}
\item{permuting the entries of each vector according to the same permutation,}
\item{multiplying all entries of a single vector by a complex constant of magnitude 1,}
\item{multiplying the same entry of each vector by a complex constant of magnitude 1.}
\end{enumerate}
Under a suitable combination of these operations, we can transform Hoggar's 64 lines (as interpreted above) into a form that exposes a new link with MUBs. 

\begin{ex}  Let $B_1,B_2,B_3,B_4$ be the $4$ MUBs of Example~$\ref{EX:MUBs4}$. Let $C_j$ be the $4\times 4$ matrix with $-1+i$ in the $j$-th column and zeros elsewhere. Then Hoggar's $64$ lines in $\C^8$ are equivalent to
\begin{displaymath}
\begin{array}{@{[}cc@{]\hspace{15pt}[}cc@{]\hspace{15pt}[}cc@{]\hspace{15pt}[}cc@{]}}
B_1 & C_1 & B_1 & -C_1 & -C_1 & B_1 & C_1 & B_1 \\
B_2 & C_2 & B_2 & -C_2 & -C_2 & B_2 & C_2 & B_2 \\
B_3 & C_3 & B_3 & -C_3 & -C_3 & B_3 & C_3 & B_3 \\
B_4 & C_4 & B_4 & -C_4 & -C_4 & B_4 & C_4 & B_4 \\
\end{array}
\end{displaymath}
where $[B\;\,C]$ represents the set of vectors which are the concatenation of corresponding vectors in $B$ and $C$. 
\label{EX:Hoggartwist} \end{ex}

In Example~\ref{EX:Hoggartwist} we can see how the 4 MUBs of Example~\ref{EX:MUBs4} are embedded in a set of equiangular lines equivalent to Hoggar's. In this way, we can view Example~\ref{EX:Hoggartwist} as constructing 64 equiangular lines in $\C^8$ from 4 MUBs in $\C^4$. In fact, it is just one instance of the following construction of a one-parameter family of sets of 64 equiangular lines in $\C^8$. 

Let $a$ be a real parameter. Let $C_j(a)$ be the $4\times 4$ matrix with $\frac{a-1+i(a+1)}{\sqrt{1+a^2}}$ in column $j$ and zeros elsewhere and $D_j(a)$ be the $4\times 4$ matrix with $\frac{a+1+i(a-1)}{\sqrt{1+a^2}}$ in column $j$ and zeros elsewhere. Then one can verify using a computer algebra system, or even by hand, that the following vectors are a set of 64 equiangular lines in $\C^8$ for all values of $a\in\R$:
\begin{displaymath}
\begin{array}{@{[}cc@{]\hspace{15pt}[}cc@{]\hspace{15pt}[}cc@{]\hspace{15pt}[}cc@{}l}
B_1 & C_1(a) & B_1 & -C_1(a) & D_1(a) & B_1 & -D_1(a) & B_1&] \\
B_2 & C_2(a) & B_2 & -C_2(a) & D_2(a) & B_2 & -D_2(a) & B_2&] \\
B_3 & C_3(a) & B_3 & -C_3(a) & D_3(a) & B_3 & -D_3(a) & B_3&] \\
B_4 & C_4(a) & B_4 & -C_4(a) & D_4(a) & B_4 & -D_4(a) & B_4&]. \\
\end{array}
\end{displaymath}
The lines of Example~\ref{EX:Hoggartwist} can be obtained by setting $a=0$.

Notice that in contrast to the construction of \S\ref{S:LinesfromMUBs}, this construction involves a change of dimension, namely using MUBs in $\C^4$ to construct equiangular lines in $\C^8$. It is then natural to ask the following question:
\begin{question} Can we construct lines having similar form to those in Example~$\ref{EX:Hoggartwist}$ in dimensions other than $8$? \label{Q:Hoggar} \end{question}
And more generally,
\begin{question} Can we construct $D^2$ equiangular lines in $\C^D$ from $d$ MUBs in $\C^d$ for $D\neq d$? \label{Q:MUBstoLines2}\end{question}

\section{Construction 3}\label{S:Lblocks}

In our final construction, we suggest another approach to answering Question~\ref{Q:MUBstoLines2}. We will show how to join together blocks of the form $L_d^R(\pi,v)$ as constructed in \S\ref{S:LinesfromMUBs} to give a set of $d^2$ equiangular lines in $\C^{2d}$ for infinitely many values of $d$. Furthermore, we will show how in dimension 8 we can extend this set of 16 lines to a maximum-sized set of 64 equiangular lines. 

As described in \S\ref{S:LinesfromMUBs}, the vectors in $L^R_d(\pi,v)$ are derived from a set of $d$ MUBs constructed from a $(d,d,d,1)$-RDS in an abelian group. Write $[L^R_d(\pi,v) \hspace{5pt} L^R_d(\pi,v')]$ for the set of vectors in which each vector is the concatenation of corresponding vectors in $L^R_d(\pi,v)$ and $L^R_d(\pi,v')$.

\begin{ex} 
$d=4:$\! Take $R$ to be the RDS of Example~$\ref{EX:RDS4}$. Let $\pi = [1,3,4,2]$, and let $v,v'\in\C$. Construct $L^R_4(\pi,v)$ and $L^R_4(\pi,v')$ as in $(\ref{EQ:Lblock})$ (see Example~$\ref{EX:MUBLines4}$). Then
\begin{displaymath}
[L^R_4(\pi,v)\hspace{5pt} L^R_4(\pi,v')] = \left\{
\begin{array}{@{\;(}cccccccc@{)\;}}
v & 1 & 1 & 1 & v' & 1 & 1 & 1 \\
v & 1 & -1 & -1 & v' & 1 & -1 & -1 \\
v & -1 & 1 & -1 & v' & -1 & 1 & -1 \\
v & -1 & -1 & 1 & v' & -1 & -1 & 1 \\[8pt]
1 & 1 & iv & -i & 1 & 1 & iv' & -i \\
1 & 1 & -iv & i & 1 & 1 & -iv' & i \\
1 & -1 & iv & i & 1 & -1 & iv' & i \\
1 & -1 & -iv & -i & 1 & -1 & -iv' & -i\\[8pt]
1 & i & 1 & -iv & 1 & i & 1 & -iv' \\
1 & i & -1 & iv & 1 & i & -1 & iv' \\
1 & -i & 1 & iv & 1 & -i & 1 & iv' \\
1 & -i & -1 & -iv & 1 & -i & -1 & -iv' \\[8pt]
1 & iv & i & -1 & 1 & iv' & i & -1 \\
1 & iv & -i & 1 & 1 & iv' & -i & 1 \\
1 & -iv & i & 1 & 1 & -iv' & i & 1 \\
1 & -iv & -i & -1 & 1 & -iv' & -i & -1 
\end{array} \right\}. \end{displaymath}
\label{EX:2L4blocks}
\end{ex}

\begin{lemma} Let $R$ be a $(d,d,d,1)$-RDS in an abelian group, let $\pi\in S_d$ and let $a,b\in\R$. Then all inner products between distinct vectors of the $d^2$ vectors of $[L^R_{d}(\pi,a+ib) \hspace{5pt} L^R_d(\pi,2-a-ib)]$ in $\C^{2d}$ have magnitude ${2}({b^2+(a-1)^2})$ or $2\sqrt{d}$. \label{LEM:2Lblockangles} \end{lemma}

\begin{proof} 
We consider two cases, according to whether distinct vectors of $L_d^R(\pi,v)$ originate from the same basis or from distinct bases. 

In the first case, consider the inner product of distinct vectors of $L^R_d(\pi,v)$ constructed from vectors in the same basis $B^R_j$. Since the original vectors are orthogonal, this inner product is $\xi(|v|^2-1)$ for some root of unity~$\xi$. When $v=a+ib$, the inner product becomes $\xi(a^2+b^2-1)$ and when $v=2-a-ib$ it becomes $\xi((2-a)^2+b^2-1)$. Thus the corresponding concatenated vectors have inner product $\xi(a^2+b^2-1+(2-a)^2+b^2-1)$, which has magnitude ${2}({b^2+(a-1)^2})$. 

In the second case, consider vectors of $L^R_d(\pi,v)$ constructed from vectors in distinct bases $B^R_j, B^R_k$ . Let these constructed vectors be given by 
\begin{equation}
\begin{array}{rcccccccc}
\x & = & (x_{1} & x_2 & \ldots & \ldots & vx_{\pi(j)} & \ldots & x_d)\\
\y & = & (y_{1} & y_2 & \ldots & vy_{\pi(k)} & \ldots & \ldots & y_d). 
\end{array}
\end{equation}
When $v=1$, by construction all of the entries $x_\ell$, $y_\ell$ are roots of unity (see Theorem~\ref{THM:RDStoMUBs}) and so each vector has norm $\sqrt{d}$. Therefore the inner product $\sum_{\ell=1}^d x_{\ell}\overline{y_\ell}$ of these vectors when $v=1$ has magnitude $\sqrt{d}$, by (\ref{EQ:MUBdef}).
Now the inner product of $\x$ and $\y$ in $L^R_d(\pi,v)$ is 
\begin{equation*}
\begin{split} x_1\overline{y_1} + \cdots + vx_{\pi(j)}\overline{y_{\pi(j)}}+\cdots+\overline{v}x_{\pi(k)}\overline{y_{\pi(k)}}+\cdots+x_d\overline{y_d}  \\
= \sum_{\ell=1}^d x_{\ell}\overline{y_\ell} + (v-1)x_{\pi(j)}\overline{y_{\pi(j)}}+(\overline{v}-1)x_{\pi(k)}\overline{y_{\pi(k)}}.
\end{split}
\end{equation*}
This means that the corresponding concatenated vectors have inner product 
\begin{equation*}\begin{split}
\sum_{\ell=1}^d x_{\ell}\overline{y_\ell} + (a+ib-1)x_{\pi(j)}\overline{y_{\pi(j)}}+(a-ib-1)x_{\pi(k)}\overline{y_{\pi(k)}}\hspace{60pt} \\
+\;\sum_{\ell=1}^d x_{\ell}\overline{y_\ell} + (2-a-ib-1)x_{\pi(j)}\overline{y_{\pi(j)}}+(2-a+ib-1)x_{\pi(k)}\overline{y_{\pi(k)}} \\
=2\sum_{\ell=1}^d x_{\ell}\overline{y_\ell}.
\end{split} \end{equation*}
Since $\left|\sum_{\ell=1}^d x_{\ell}\overline{y_\ell}\right|=\sqrt{d}$, the concatenated vectors have inner product of magnitude $2\sqrt{d}$.
\end{proof}

\begin{cor} Let $R$ be a $(d,d,d,1)$-RDS in an abelian group, and let $\pi\in S_d$.  Then there are infinitely many choices of $a,b \in \R$ so that the vectors of $[L^R_{d}(\pi,a+ib) \hspace{5pt} L^R_d(\pi,2-a-ib)]$ form a set of $d^2$ equiangular lines in $\C^{2d}$. \label{COR:Lblocks} \end{cor}

Corollary~\ref{COR:Lblocks} follows from Lemma~\ref{LEM:2Lblockangles}, since for every dimension $d$ there are infinitely many choices of $a,b$ such that ${2}(b^2+(a-1)^2)=2\sqrt{d}$. Thus Corollary~\ref{COR:Lblocks} gives $d^2$ equiangular lines in $\C^{2d}$ whenever $d=p^r$ for some prime $p$, as we can construct blocks of the form $L_d^R(\pi,v)$ in these dimensions. So the list of dimensions for which one can construct $\Theta(d^2)$ equiangular lines in $\C^d$, which was previously known to include $d=3\cdot 2^{2t-1}-1$ \cite{deCaen}, and $d=p^r+1$ for $p$ prime \cite{Godsil-Roy, Konig}, can now be extended to include $d=2p^r$ for $p$ prime\footnote{A function $f$ from $\N$ to $\R^+$ is $\Theta(d^2)$ if there are positive constants $c$ and $C$, independent of $d$, for which $c d^2 \le f(d) \le C d^2$ for all sufficiently large~$d$.}. 

Notice that the angle between each pair of distinct lines in Corollary~\ref{COR:Lblocks} is $\frac{1}{1+\sqrt{d}}$ (as is easily checked by normalizing these vectors).
Using the special bounds calculated by Delsarte, Goethals and Seidel (\cite[Table I]{DGS-bounds} with $\alpha=\beta=\frac{1}{1+\sqrt{d}}$ and $n=2d$ over $\C$), we find that the function
\begin{equation*} 
f(d) = \frac{d(2d+1)(2\sqrt{d}+d)^2}{d^2+4d+2\sqrt{d}} \end{equation*}
is an upper bound on the number of vectors in $\C^{2d}$ having angle $\frac{1}{1+\sqrt{d}}$ between each pair of distinct vectors. In the range $d \ge 1$ we have $2d^2 < f(d) \le 4d^2$, and the larger value $4d^2$ is attained exactly at $d=4$ and the smaller value $2d^2$ is the asymptotic value of $f(d)$.
Since Corollary~\ref{COR:Lblocks} gives only $d^2$ equiangular lines with the specified angle $\frac{1}{1+\sqrt{d}}$, we ask the following questions:

\begin{question} 
Can we extend the set of $d^2$ equiangular lines in $\C^{2d}$ given in Corollary~$\ref{COR:Lblocks}$ by adding some or all of the vectors of additional blocks of the form $L_d^R(\pi,v)$ for suitable $\pi$ and $v$, and if so what is the largest possible number of resulting equiangular lines? \label{Q:LblockExtend}
\end{question}

\begin{question} For large $d$, can we achieve the asymptotic bound of $2d^2$ equiangular lines having angle $\frac{1}{1+\sqrt{d}}$ in $\C^{2d}$? \label{Q:LblockAsymp} \end{question}

The value $d=4$ in Question~\ref{Q:LblockExtend} is of special interest, because it is the only value of $d \ge 1$ for which $f(d) = 4d^2$ (so that there is a possibility of extending the $d^2$ equiangular lines in $\C^{2d}$ to a maximum-sized set of size $4d^2$. We can alternatively identify the candidate value $d=4$ by equating the specified angle $\frac{1}{1+\sqrt{d}}$ with the angle $\frac{1}{\sqrt{2d+1}}$ given by Proposition~\ref{PROP:angle}.)
It turns out that when $d=4$ we can indeed combine several blocks of the form $L_4^d(\pi,v)$ to form a set of 64 equiangular lines in $\C^8$. 

\begin{ex} Let $R$ be the RDS of Example~$\ref{EX:RDS4}$ and let $\pi=[1,3,4,2]$. The following is a set of $64$ equiangular lines in $\C^8$:
$${\renewcommand{\arraystretch}{1.1}
\begin{array}{@{[}r@{(\pi,\;}r@{\,)\, \hspace{5pt}}r@{(\pi,\; }r@{\,)\,]\hspace{10pt}[}r@{(\pi,\;}r@{\,)\, \hspace{5pt}}r@{(\pi,\; }r@{\,)\,]}}
L^R_4 & 2+i 	& L^R_4 & -i 
& L^R_4 & -i	& L^R_4 & 2+i  \\[2pt]
L^R_4 & -1+2i 	& -L^R_4 & 1 
& L^R_4& 1	& -L^R_4 &-1+2i
\end{array}
}$$
The lines are given explicitly by the following vectors:
 $$ \fontsize{9}{9.75}\selectfont
\begin {array}{@{(}*{4}{@{\!}c@{}}*{4}{@{\,}c@{\;}}@{\!)}} 2+i&1&1&1&-i&1&1&1
\\ \noalign{\medskip}2+i&1&-1&-1&-i&1&-1&-1\\ \noalign{\medskip}2+i&-1
&1&-1&-i&-1&1&-1\\ \noalign{\medskip}2+i&-1&-1&1&-i&-1&-1&1
\\ \noalign{\medskip}1&1&-1+2\,i&-i&1&1&1&-i\\ \noalign{\medskip}1&1&1
-2\,i&i&1&1&-1&i\\ \noalign{\medskip}1&-1&-1+2\,i&i&1&-1&1&i
\\ \noalign{\medskip}1&-1&1-2\,i&-i&1&-1&-1&-i\\ \noalign{\medskip}1&i
&1&1-2\,i&1&i&1&-1\\ \noalign{\medskip}1&i&-1&-1+2\,i&1&i&-1&1
\\ \noalign{\medskip}1&-i&1&-1+2\,i&1&-i&1&1\\ \noalign{\medskip}1&-i&
-1&1-2\,i&1&-i&-1&-1\\ \noalign{\medskip}1&-1+2\,i&i&-1&1&1&i&-1
\\ \noalign{\medskip}1&-1+2\,i&-i&1&1&1&-i&1\\ \noalign{\medskip}1&1-2
\,i&i&1&1&-1&i&1\\ \noalign{\medskip}1&1-2\,i&-i&-1&1&-1&-i&-1
\\[15pt]
-1+2\,i&1&1&1&-1&-1&-1&-1
\\ \noalign{\medskip}-1+2\,i&1&-1&-1&-1&-1&1&1\\ \noalign{\medskip}-1+
2\,i&-1&1&-1&-1&1&-1&1\\ \noalign{\medskip}-1+2\,i&-1&-1&1&-1&1&1&-1
\\ \noalign{\medskip}1&1&-2-i&-i&-1&-1&-i&i\\ \noalign{\medskip}1&1&2+
i&i&-1&-1&i&-i\\ \noalign{\medskip}1&-1&-2-i&i&-1&1&-i&-i
\\ \noalign{\medskip}1&-1&2+i&-i&-1&1&i&i\\ \noalign{\medskip}1&i&1&2+
i&-1&-i&-1&i\\ \noalign{\medskip}1&i&-1&-2-i&-1&-i&1&-i
\\ \noalign{\medskip}1&-i&1&-2-i&-1&i&-1&-i\\ \noalign{\medskip}1&-i&-
1&2+i&-1&i&1&i\\ \noalign{\medskip}1&-2-i&i&-1&-1&-i&-i&1
\\ \noalign{\medskip}1&-2-i&-i&1&-1&-i&i&-1\\ \noalign{\medskip}1&2+i&
i&1&-1&i&-i&-1\\ \noalign{\medskip}1&2+i&-i&-1&-1&i&i&1
\end {array} 
\hspace{12pt}
\begin {array}{@{(}*{4}{@{\,}c@{\;\,}}*{4}{@{\!}c@{}}@{)}}  -i&1&1&1&2+i&1&1&1
\\ \noalign{\medskip}-i&1&-1&-1&2+i&1&-1&-1\\ \noalign{\medskip}-i&-1&
1&-1&2+i&-1&1&-1\\ \noalign{\medskip}-i&-1&-1&1&2+i&-1&-1&1
\\ \noalign{\medskip}1&1&1&-i&1&1&-1+2\,i&-i\\ \noalign{\medskip}1&1&-
1&i&1&1&1-2\,i&i\\ \noalign{\medskip}1&-1&1&i&1&-1&-1+2\,i&i
\\ \noalign{\medskip}1&-1&-1&-i&1&-1&1-2\,i&-i\\ \noalign{\medskip}1&i
&1&-1&1&i&1&1-2\,i\\ \noalign{\medskip}1&i&-1&1&1&i&-1&-1+2\,i
\\ \noalign{\medskip}1&-i&1&1&1&-i&1&-1+2\,i\\ \noalign{\medskip}1&-i&
-1&-1&1&-i&-1&1-2\,i\\ \noalign{\medskip}1&1&i&-1&1&-1+2\,i&i&-1
\\ \noalign{\medskip}1&1&-i&1&1&-1+2\,i&-i&1\\ \noalign{\medskip}1&-1&
i&1&1&1-2\,i&i&1\\ \noalign{\medskip}1&-1&-i&-1&1&1-2\,i&-i&-1
\\[15pt]
1&1&1&1&1-2\,i&-1&-1&-1
\\ \noalign{\medskip}1&1&-1&-1&1-2\,i&-1&1&1\\ \noalign{\medskip}1&-1&
1&-1&1-2\,i&1&-1&1\\ \noalign{\medskip}1&-1&-1&1&1-2\,i&1&1&-1
\\ \noalign{\medskip}1&1&i&-i&-1&-1&2+i&i\\ \noalign{\medskip}1&1&-i&i
&-1&-1&-2-i&-i\\ \noalign{\medskip}1&-1&i&i&-1&1&2+i&-i
\\ \noalign{\medskip}1&-1&-i&-i&-1&1&-2-i&i\\ \noalign{\medskip}1&i&1&
-i&-1&-i&-1&-2-i\\ \noalign{\medskip}1&i&-1&i&-1&-i&1&2+i
\\ \noalign{\medskip}1&-i&1&i&-1&i&-1&2+i\\ \noalign{\medskip}1&-i&-1&
-i&-1&i&1&-2-i\\ \noalign{\medskip}1&i&i&-1&-1&2+i&-i&1
\\ \noalign{\medskip}1&i&-i&1&-1&2+i&i&-1\\ \noalign{\medskip}1&-i&i&1
&-1&-2-i&-i&-1\\ \noalign{\medskip}1&-i&-i&-1&-1&-2-i&i&1
\end {array}
$$
Notice that this is also an example of almost flat equiangular lines. \label{EX:LblockLines}
\end{ex}

It is not the case that for every choice of permutation $\pi$, the set $[L_4^R(\pi,a+ib)\;L_4^R(\pi,2-a-ib)]$ can be extended to 64 equiangular lines in $\C^8$. In fact, the ability to extend the set of Lemma~\ref{LEM:2Lblockangles} to a set of maximum size is highly sensitive to the choice of additional blocks. Notice also that in Example~\ref{EX:LblockLines}, the lower blocks of vectors do not follow the exact structure of Lemma~\ref{LEM:2Lblockangles}, but instead are of the form $[L_4^R(\pi,i(a+ib))\;-L_4^R(\pi,i(2-a-ib))]$.
These subtle differences indicate that answering Question~\ref{Q:LblockExtend} (for $d > 4$) and Question~\ref{Q:LblockAsymp} might involve careful parameter choices as well as small variations in the form of the blocks. 
This discussion also motivates a final question:

\begin{question} 
Is there a variant of the block construction of Lemma~$\ref{LEM:2Lblockangles}$ in $\C^d$ from which we can construct $d^2$ equiangular lines in $\C^D$ with angle $\frac{1}{\sqrt{D+1}}$ for some $D > d$ (so that the resulting lines have the correct angle required for a maximum-sized set of size $D^2$)? 
\label{Q:AlterLblocks} \end{question}

\section{Conclusion} \label{S:conclusion}

We have seen that MUBs and sets of equiangular lines are more deeply intertwined than previously recognized. We believe that the new constructions presented here, and the questions posed, open new avenues for exploring the existence of maximum-sized sets of equiangular lines.

\bibliographystyle{plain}
\bibliography{mubs-lines}

\begin{thebibliography}{10}

\bibitem{Appleby-Clifford}
D.~M. Appleby.
\newblock Symmetric informationally complete-positive operator valued measures
  and the extended {C}lifford group.
\newblock {\em J. Math. Phys.}, 46(5):052107 (29 pages), 2005.

\bibitem{Appleby-prime}
D.~M. Appleby.
\newblock S{IC}-{POVMS} and {MUBS}: {G}eometrical relationships in prime
  dimension.
\newblock In {\em Foundations of Probability and Physics---5}, volume 1101 of
  {\em AIP Conf. Proc.}, pages 223--232. Amer. Inst. Phys., New York, 2009.

\bibitem{Appleby-imprimitivity}
D.~M. Appleby, I.~Bengtsson, S.~Brierley, {\AA}.~Ericsson, M.~Grassl, and
  J.-{\AA}. Larsson.
\newblock Systems of imprimitivity for the {C}lifford group.
\newblock {\em Quantum Inf. Comput.}, 14(3--4):339--360, 2014.

\bibitem{Appleby-monomial}
D.~M. Appleby, I.~Bengtsson, S.~Brierley, M.~Grassl, D.~Gross, and J.-{\AA}.
  Larsson.
\newblock The monomial representations of the {C}lifford group.
\newblock {\em Quantum Inf. Comput.}, 12(5--6):404--431, 2012.

\bibitem{Appleby-abstract}
M.~Appleby.
\newblock {SIC-POVM}s, theta functions and squeezed states.
\newblock Abstract for 2010--2011 Clifford Lectures, Tulane University,
  http://tulane.edu/sse/math/news/clifford-lectures-2011.cfm.

\bibitem{Bandyopadhyay}
S.~Bandyopadhyay, P.~O. Boykin, V.~Roychowdhury, and F.~Vatan.
\newblock A new proof for the existence of mutually unbiased bases.
\newblock {\em Algorithmica}, 34(4):512--528, 2002.

\bibitem{Belovs-thesis}
A.~Belovs.
\newblock Welch bounds and quantum state tomography.
\newblock Master's thesis, University of Waterloo, 2008.

\bibitem{Beneduci-operational}
R.~Beneduci, T.~J. Bullock, P.~Busch, C.~Carmeli, T.~Heinosaari, and A.~Toigo.
\newblock Operational link between mutually unbiased bases and symmetric
  informationally complete positive operator-valued measures.
\newblock {\em Phys. Rev. A}, 88:032312 (15 pages), 2013.

\bibitem{Bengtsson-three-ways}
I.~Bengtsson.
\newblock Three ways to look at mutually unbiased bases.
\newblock In {\em Foundations of Probability and Physics --- 4}, volume 889 of
  {\em AIP Conf. Proc.}, pages 40--51. Amer. Inst. Phys., New York, 2007.

\bibitem{Bengtsson-Eddington}
I.~Bengtsson.
\newblock From {SIC}s and {MUB}s to {E}ddington.
\newblock {\em J. Phys. Conf. Ser.}, 254:012007 (12 pages), 2010.

\bibitem{Bengtsson-KS}
I.~Bengtsson, K.~Blanchfield, and A.~Cabello.
\newblock A {K}ochen-{S}pecker inequality from a {SIC}.
\newblock {\em Phys. Lett. A}, 376(4):374--376, 2012.

\bibitem{deCaen}
D.~de~Caen.
\newblock Large equiangular sets of lines in {E}uclidean space.
\newblock {\em Electron. J. Combin.}, 7:\#R55 (3 pages), 2000.

\bibitem{DGS-bounds}
P.~Delsarte, J.~M. Goethals, and J.~J. Seidel.
\newblock Bounds for systems of lines, and {J}acobi polynomials.
\newblock {\em Philips Res. Repts}, 30:91--105, 1975.

\bibitem{Godsil-slides}
C.~Godsil.
\newblock Quantum geometry: {MUB}'s and {SIC-POVM}'s, December 2009.
\newblock Slide presentation, http://quoll.uwaterloo.ca/pdfs/perth.pdf.

\bibitem{Godsil-Roy}
C.~Godsil and A.~Roy.
\newblock Equiangular lines, mutually unbiased bases, and spin models.
\newblock {\em European J. Combin.}, 30(1):246--262, 2009.

\bibitem{Grassl-dim6}
M.~Grassl.
\newblock On {SIC}-{POVM}s and {MUB}s in dimension 6.
\newblock In {\em Proceedings {ERATO} {C}onference on {Q}uantum {I}nformation
  {S}cience 2004}, pages 60--61, Tokyo, 2004.

\bibitem{Grassl-tomography}
M.~Grassl.
\newblock Tomography of quantum states in small dimensions.
\newblock In {\em Proceedings of the {W}orkshop on {D}iscrete {T}omography and
  its {A}pplications}, volume~20 of {\em Electron. Notes Discrete Math.}, pages
  151--164. Elsevier, Amsterdam, 2005.

\bibitem{Grassl-slides}
M.~Grassl.
\newblock Finding equiangular lines in complex space, July 2006.
\newblock Slide presentation at MAGMA 2006 Conference, Technische Universit\"at
  Berlin,
  http://magma.maths.usyd.edu.au/conferences/Magma2006/talks/Grassl\_Berlin.pdf.

\bibitem{Grassl-computing}
M.~Grassl.
\newblock Computing equiangular lines in complex space.
\newblock In {\em Mathematical {M}ethods in {C}omputer {S}cience}, volume 5393
  of {\em Lecture Notes in Comput. Sci.}, pages 89--104. Springer, Berlin,
  2008.

\bibitem{Greaves}
G.~Greaves, J.~H. Koolen, A.~Munemasa, and F.~Sz\"{o}ll\H{o}si.
\newblock Equiangular lines in {E}uclidean spaces.
\newblock arXiv:1403.2155 [math.CO].

\bibitem{Haantjes}
J.~Haantjes.
\newblock Equilateral point-sets in elliptic two- and three-dimensional spaces.
\newblock {\em Nieuw Arch. Wiskunde (2)}, 22:355--362, 1948.

\bibitem{Hoggar-two}
S.~G. Hoggar.
\newblock Two quaternionic $4$-polytopes.
\newblock In {\em The {G}eometric {V}ein}, pages 219--230. Springer, New
  York-Berlin, 1981.

\bibitem{Hoggar-t-designs}
S.~G. Hoggar.
\newblock {$t$}-designs in projective spaces.
\newblock {\em European J. Combin.}, 3(3):233--254, 1982.

\bibitem{Hoggar-64}
S.~G. Hoggar.
\newblock $64$ lines from a quaternionic polytope.
\newblock {\em Geom. Dedicata}, 69(3):287--289, 1998.

\bibitem{Jedwab-Wiebe-simple}
J.~Jedwab and A.~Wiebe.
\newblock A simple construction of complex equiangular lines.
\newblock In C.~J. Colbourn, editor, {\em Algebraic Design Theory and Hadamard
  Matrices}, Springer Proc. Math. Stat. Springer, 2015.
\newblock To appear. arXiv:1408.2492 [math.CO].

\bibitem{Khatirinejad}
M.~Khatirinejad.
\newblock On {W}eyl-{H}eisenberg orbits of equiangular lines.
\newblock {\em J. Algebraic Combin.}, 28(3):333--349, 2008.

\bibitem{Kibler}
M.~R. Kibler.
\newblock On two ways to look for mutually unbiased bases.
\newblock {\em Acta Polytechnica}, 54(2):124--126, 2014.

\bibitem{Konig}
H.~K{\"o}nig.
\newblock Cubature formulas on spheres.
\newblock In {\em Advances in {M}ultivariate {A}pproximation
  ({W}itten-{B}ommerholz, 1998)}, volume 107 of {\em Math. Res.}, pages
  201--211. Wiley-VCH, Berlin, 1999.

\bibitem{Pott}
A.~Pott.
\newblock {\em Finite {G}eometry and {C}haracter {T}heory}, volume 1601 of {\em
  Lecture Notes in Mathematics}.
\newblock Springer-Verlag, Berlin, 1995.

\bibitem{Renes}
J.~M. Renes, R.~Blume-Kohout, A.~J. Scott, and C.~M. Caves.
\newblock Symmetric informationally complete quantum measurements.
\newblock {\em J. Math. Phys.}, 45(6):2171--2180, 2004.

\bibitem{Roy-thesis}
A.~Roy.
\newblock {\em Complex Lines with Restricted Angles}.
\newblock PhD thesis, University of Waterloo, 2005.

\bibitem{Scott-Grassl}
A.~J. Scott and M.~Grassl.
\newblock Symmetric informationally complete positive-operator-valued measures:
  a new computer study.
\newblock {\em J. Math. Phys.}, 51(4):042203 (16 pages), 2010.

\bibitem{Weyl}
H.~Weyl.
\newblock {\em The Theory of Groups and Quantum Mechanics}.
\newblock Dover Publications, New York, 1950.

\bibitem{Wiebe-thesis}
A.~Wiebe.
\newblock Constructions of complex equiangular lines.
\newblock Master's thesis, Simon Fraser University, 2013.

\bibitem{Wootters}
W.~K. Wootters.
\newblock Quantum measurements and finite geometry.
\newblock {\em Found. Phys.}, 36(1):112--126, 2006.

\bibitem{Zauner-thesis}
G.~Zauner.
\newblock {\em Quantendesigns: Grundz\"uge einer nichtkommutativen
  Designtheorie}.
\newblock PhD thesis, University of Vienna, 1999.

\end{thebibliography}

\end{document}